\documentclass{article}
\usepackage{graphicx} 
\usepackage{algorithm}
\usepackage{xcolor}
\usepackage{cite}
\usepackage{amsthm}
\usepackage{svg}
\usepackage{booktabs} 
\usepackage{algpseudocode}
\usepackage{textcomp}
\usepackage{amsfonts}
\usepackage{subcaption}
\usepackage{multicol}
\usepackage{soul}
\usepackage{amsmath,mathtools,amssymb,amsfonts}
\usepackage{nicematrix}

\usepackage{mathtools,amssymb,lipsum}

\usepackage{cuted}
\newcommand{\norm}[1]{\left\lVert#1\right\rVert}
\newcommand{\xddots}{%
  \raise 4pt \hbox {.}
  \mkern 6mu
  \raise 1pt \hbox {.}
  \mkern 6mu
  \raise -2pt \hbox {.}
}
\DeclareMathOperator*{\SubjectTo}{Subject\phantom{a}to:}
\DeclareMathOperator*{\Minimize}{Minimize:}

\newtheorem{theorem}{\bf{Theorem}}

\newtheorem{remark}{Remark}

\newtheorem{assumption}{\it{Assumption}}

\title{Concurrent Parameter Learning and Current Control for Large-scale Grid-following Inverter-based Resources}
\author{Satish Vedula}
\date{September 2026}

\begin{document}

\maketitle
\textbf{COPYRIGHTED: UNDER REVIEW FOR IEEE TEXAS POWER AND ENERGY CONFERENCE, COLLEGE STATION, TX, USA, 2027}

The integration of large-scale inverter-based resources (IBRs) into the grid presents two key challenges in terms of optimal power generation and transmission system stability. This work proposes concurrent transmission parameter learning and optimal power generation and current control for the grid-following IBRs. The proposed approach estimates transmission parameters online using the current measurements, enabling the IBRs to adapt their control actions to evolving grid conditions. The learned parameter estimates are incorporated into an optimization-based power distribution strategy to determine the optimal power setpoints. In parallel, a current control scheme is designed to regulate the grid-following (GFL) IBRs current output based on the estimates and received active and reactive power references from the high-level optimizer, thereby improving the dynamic response to the grid events. The effectiveness of the proposed framework is validated through a simulation in a MATLAB-Simulink environment. From the results, the impact of the concurrent estimator and current controller during the grid events can be seen.

\section{Introduction}
Rapid growth of renewable energy resources such as solar and wind generation which are inverter-based, and their integration into the existing power grids have led to substantial structural and operational transformations in the existing electrical power systems (EPS). Existing power grids consisting of synchronous machines (SMs) are being replaced by or additionally compounded by large-scale inverter-based resources (IBRs) accounting for the addition of more distributed energy resources (DERs) to the grid. Moreover, inherently IBRs exhibit low inertia due to a lack of rotating mechanical components, in contrast to the SMs, thus responding swiftly to stochastic events. Another control challenge is to incorporate the fast-switching properties of the IBRs for controlling the IBR current. The lack of inertia and fast switching gives rise to control challenges in terms of stabilizing the EPSs in the presence of the IBRs and the overall robustness of the EPSs \cite{8450880}. In the existing literature, numerous researchers have proposed a \textit{virtual inertia} based approach to tackle the control issues arising due to lack of inertia \cite{4596800,6919271,anubi2022robust}.

Nevertheless of the inertial properties of the SMs and the IBRs, the indubitable objective remains identical in managing the active and the reactive power sharing in the power grid. An effective control method employed to address effective power sharing is \textit{droop control} \cite{5546958}. Numerous control techniques derived from droop control have been proposed in the existing literature, such as the PD-like discrete-time consensus control and distributed droop control \cite{chen2020distributed},~\cite{schiffer2015voltage}. However, the lack of robustness in these methods due to exogenous disturbances is addressed by means of the sliding mode approach \cite{alfaro2021distributed}. However, sliding mode controllers induce the phenomenon of chattering, which can be reduced by means of an estimator and chattering attenuation designs \cite{9829024}. Consensus-based algorithms provide efficient frequency regulation; nonetheless, they suffer from communication costs and reliability of communication links \cite{7478670}. A decentralized transmission-line health-based active power tracking is proposed in \cite{vedula2024faulttolerantdecentralizedcontrollargescale} to address the problem of power tracking under transmission faults. 

Traditional power grid control is multilayered and hierarchical in structure, often comprising a droop-based primary layer, power and frequency deviation mitigation in the secondary layer, and power flow management in the tertiary layer \cite{5546958}. The goal of an optimal power flow (OPF) problem is to minimize the generation losses or generation costs, thus optimally managing the power supplied by the SMs or IBRs. In \cite{6472268}, a centralized two-layered (schedule layer and dispatch layer) energy management approach for two modes of power grid operation, namely \textit{grid connected}, and \textit{islanded}, was proposed. While existing approaches focus dominantly on the optimal power flow problem holistically, the core idea for this work is co-designing optimal power distribution alongside the current controller.

The key contributions of the work are:
\begin{enumerate}
    \item A concurrent current control and parameter estimation strategy relaxing the requirement for persistent excitation is presented.
    \item A parameter-estimate-informed optimal optimization framework is presented for optimal and grid event-aware power setpoint distribution
\end{enumerate}

\begin{figure}[t!] 
\centerline{\includegraphics[width=.7\textwidth]{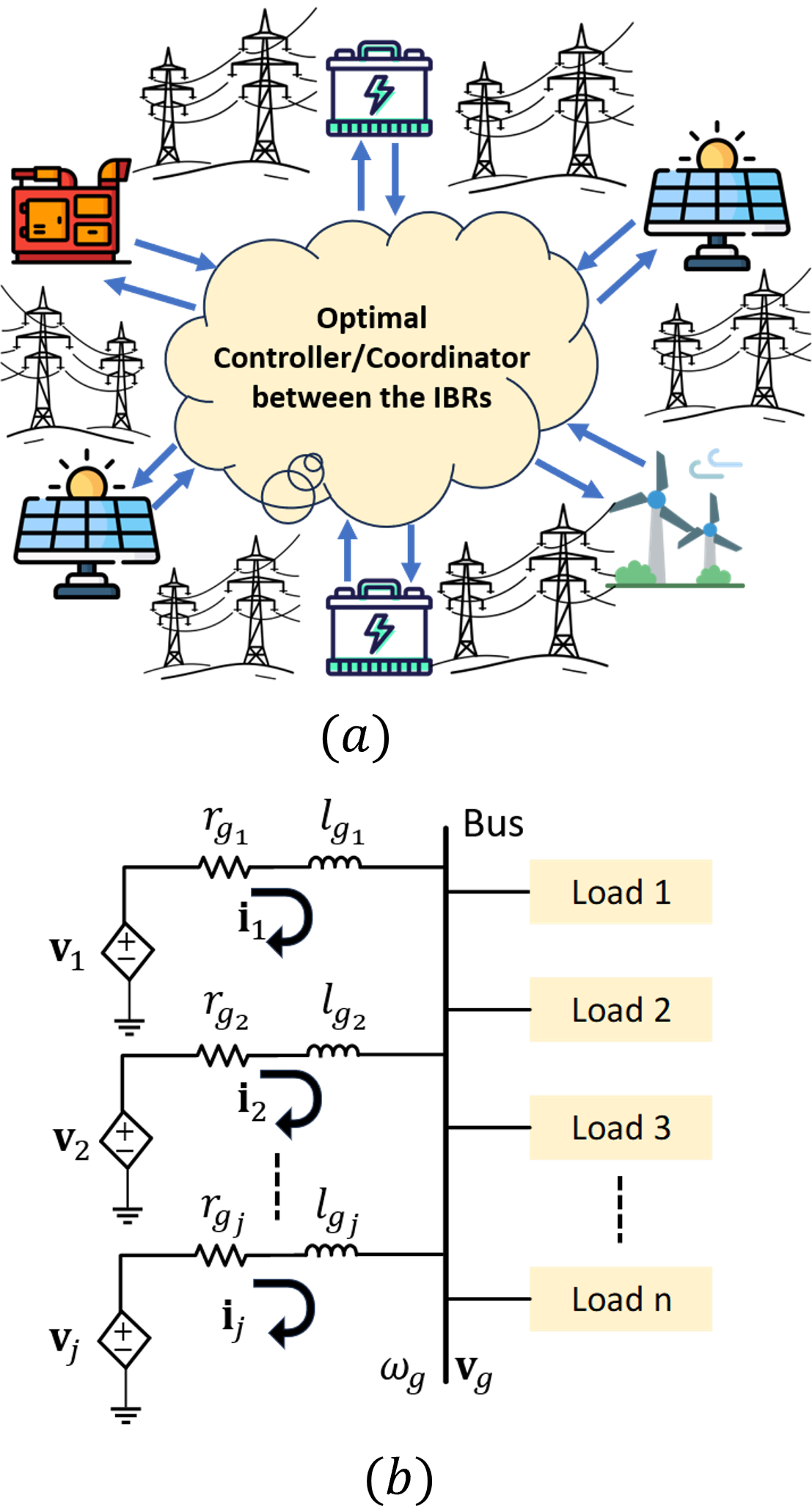}}
\caption{(a) Depiction of microgrids with IBRs, (b) Electrical equivalent modeling of the distributed energy resources.}
\label{fig_example5.jpg}
\end{figure}

The effectiveness of the proposed framework is validated using MATLAB/Simulink. The results show that the proposed approach can accurately track the parameters and accordingly drive the optimal power generation and current control actions. 

\section{Notations}
$\mathbb{N}$, $\mathbb{R}$, $\mathbb{C}$, and $\mathbb{R}_+$ denote the set of natural, real, complex, and positive real numbers. $\mathcal{L}_2$ and $\mathcal{L}_{\infty}$ denote the square-integrable (measurable) and bounded signal spaces. A real matrix with $n$ rows and $m$ columns is denoted as $X \in \mathbb{R}^{n \times m}$. The identity matrix is denoted as $I$. $X^\top$ denotes the transpose of a matrix $X$. $\textsf{trace}(X)$ denotes sum of all the diagonal elements of $X$. $\textsf{det}(X)$ denotes the determinant of $X$. Natural and Real scalars are denoted by lowercase letters (for example, $x \in \mathbb{N}$ and $y \in \mathbb{R}$). The real vectors are represented by the lowercase bold letters (i.e., $\textbf{x} \in \mathbb{R}^{n}$). The vector of ones and zeros is denoted as $\mathbf{1}$ and $\underline{\mathbf{0}}$. 
 For any vector $\mathbf{x} \in \mathbb{R}^n$, $\|\mathbf{x}\|_2 \triangleq \sqrt{\mathbf{x}^\top\mathbf{x}}$ and $\|\mathbf{x}\|_1
 \triangleq \sum_{i=1}^{n}|\mathbf{x}_i|$, representing the 2-norm and the 1-norm, respectively (where $|.|$ denotes absolute value). The dot product/inner product of the two vectors $\mathbf{x} \text{ and } \mathbf{y} \in \mathbb{R}^n$ is denoted as $\mathbf{x}^\top \mathbf{y}$. If $n=2$, the cross product/outer product is denoted as $\mathbf{x}^\top J \mathbf{y}$, where $J \triangleq\begin{bmatrix}
     0 & 1 \\ -1 & 0
 \end{bmatrix}$. $\{.\}_{\mathcal{X}}$ denotes the projection onto set $\mathcal{X}$.

 \section{Inverter Modeling and State of Charge Dynamics}
 The dynamics of the grid following IBRs connected to the grid via a transmission line is given, in $dq$ coordinates \cite{levron}, as follows :
\begin{align}\label{Inverter_Model}
    l_{g_j}\frac{d\mathbf{i}_j}{dt} &= -(r_{g_j} I-l_{g_j} \omega_g J)\mathbf{i}_j(t) + \mathbf{v}_j(t) -\mathbf{v}_{g}, \hspace{3mm} j = 1, 2, \hdots N,
\end{align}
where $\mathbf{i}_j(t) \in \mathbb{R}^2$, $\mathbf{v}_j(t) \in \mathbb{R}^2$, and $\mathbf{v}_g \in \mathcal{L}_2$ are the $j^{th}$ inverter current, voltage (control input) and measured grid voltage, respectively. $r_{g_j} \in \mathbb{R}_+$ and $l_{g_j} \in \mathbb{R}_+$ are the resistance and the inductance of the transmission line connecting to the grid in \textsf{Ohm} and \textsf{Henry}. $\omega_g \in \mathbb{R}_+$ is the grid frequency in \textsf{rad/s}. 

From the dynamics in \eqref{Inverter_Model} it can be observed that, given a grid voltage $\mathbf{v}_{g}$, the set of all equilibrium points 
\begin{align*}
    \mathcal{S}(\mathbf{v}_{g}) \triangleq
    \bigg\{(\mathbf{i}_j^*,\mathbf{v}_j^*)\in\mathbb{R}^2 \bigg| \mathbf{v}_j^* - (r_{g_j}I-l_{g_j}\omega_g J)\mathbf{i}_j^* = \mathbf{v}_{g} \bigg\}
\end{align*}
is nonempty. 

\begin{assumption} Since the IBRs are designed in a grid-following mode, the grid frequency at the steady-state $\omega_g$ is considered to be fixed. Also, the switching dynamics of the inverter are neglected. This is reasonable for this work since practical switching typically has a DC-gain of around 1 \cite{anubi2022robust}.
\end{assumption}

\begin{assumption} At the steady-state the grid voltage $\mathbf{v}_g$ is regulated to a fixed known nominal value $\mathbf{v}_{g_0}$, since the IBRs are designed in a grid-following mode. 
\end{assumption}

\section{Concurrent Estimation and Control Development}
Since the voltage is assumed to be regulated to a known nominal value given that the IBRs are operated in grid-following mode, the power tracking problem is converted to a current tracking problem. Namely, given the admissible active and reactive power pair from the power input shaping filter $(\mathbf{p}_{{ref}_j},\mathbf{q}_{{ref}_j})$, the desired reference current is given by
\begin{equation}
\begin{bmatrix} i_{{ref}_{jd}} \\ i_{{ref}_{jq}} \end{bmatrix} \triangleq   \mathbf{i}_{{ref}_j} = \frac{\mathbf{v}_{g_0}}{\norm{\mathbf{v}_{g_0}}_2^2}\mathbf{p}_{{ref}_j}+J\frac{\mathbf{v}_{g_0}}{\norm{\mathbf{v}_{g_0}}_2^2}\mathbf{q}_{{ref}_j},
\end{equation}
Next, consider the current tracking error between the system dynamics in (\ref{Inverter_Model}) and the reference current as
\begin{equation}\label{measurement_tracking}
    \mathbf{e}_j = \mathbf{i}_j-\mathbf{i}_{{ref}_j}, \hspace{2mm} j = 1,2,...,N.
\end{equation}
The open-loop error dynamics is obtained by taking the first-time derivative
\begin{align}\label{ed_open}
    l_{g_j} \dot{\mathbf{e}}_j = -(r_{g_j}I - l_{g_j}\omega_g J)\mathbf{i}_j + \underbrace{\mathbf{v}_j - \mathbf{v}_g}_{\mathbf{u}_j}.
\end{align}
Substituting $\mathbf{i}_j = \mathbf{e}_j + \mathbf{i}_{{ref}_j}$ yields
\begin{align}\label{open-loop}
  l_{g_j} \dot{\mathbf{e}}_j = -(r_{g_j}I - l_{g_j}\omega_g J)\mathbf{e}_j -(r_{g_j}I - l_{g_j}\omega_g J)\mathbf{i}_{{ref}_j}+ \mathbf{u}_j.
\end{align}
The open-loop error dynamics in \eqref{open-loop} can be \textit{linearly parameterized} \cite{Dixon} into a known matrix of measurements, also known as the regressor matrix, and a vector of unknown parameters 
\begin{align}\label{open_loop_ed}
    l_{g_j}\dot{\mathbf{e}}_j = -(r_{g_j}I - l_{g_j}\omega_g J)\mathbf{e}_j + Y_1\boldsymbol{\theta} +\mathbf{u}_j,
\end{align}
where 
\begin{align}\label{Y_the}
    Y_1 = \begin{bmatrix}  -\mathbf{i}_{{ref}_j} & \omega_g J \mathbf{i}_{{ref}_j}     \end{bmatrix}, \hspace{3mm} \boldsymbol{\theta} = \begin{bmatrix} r_{g_j} \\ l_{g_j} \end{bmatrix}.
\end{align}
Consider the control law
\begin{align}\label{control_law}
    \mathbf{u}_j =  - k_1 \mathbf{e}_j - Y_1\boldsymbol{\hat{\theta}},
\end{align}
where $k_1$ is a tunable control gain, $\boldsymbol{\hat{\theta}}$ is the estimate of the unknown parameters. Substituting the control law \eqref{control_law} in \eqref{open-loop} yields the following closed-loop error dynamics
\begin{align}\label{closed-loop_ed}
    l_{g_j}\dot{\mathbf{e}}_j = -(r_{g_j}I - l_{g_j}\omega_g J)\mathbf{e}_j - k_1 \mathbf{e}_j + Y_1\underbrace{(\boldsymbol{\theta}-\boldsymbol{\hat{\theta}})}_{\boldsymbol{\tilde{\theta}}}.
\end{align}
Consequently, the parameter update law is designed as follows
\begin{align}\label{param_update}
 \nonumber   \boldsymbol{\dot{\hat{\theta}}} &= \gamma_1 Y_1^\top \mathbf{e}_j + \gamma_1 \bigg(\int_{t}^{t+T}Y_2(\tau) d\tau\bigg)^\top \times \\& \int_{t}^{t+T}\bigg(\mathbf{u}_j(\tau) - Y_2(\tau)\bigg)\boldsymbol{\hat{\theta}}(t)d\tau, \forall t \geq 0,T > 0,
    \end{align}
where $Y_2 = \begin{bmatrix}
        \mathbf{i}_j & \mathbf{\dot{e}}_j - \omega_g J \mathbf{i}_j
    \end{bmatrix}$ is obtained by from rearranging the open-loop error dynamics in \eqref{ed_open} as follows
\begin{align*}
    l_{g_j} \dot{\mathbf{e}}_j &= -(r_{g_j}I - l_{g_j}\omega_g J)\mathbf{i}_j +\mathbf{u}_j, \\
   \mathbf{u}_j &= \underbrace{\begin{bmatrix}
        \mathbf{i}_j & \mathbf{\dot{e}}_j - \omega_g J \mathbf{i}_j
    \end{bmatrix}}_{Y_2} \underbrace{\begin{bmatrix} r_{g_j} \\ l_{g_j} \end{bmatrix}}_{\boldsymbol{\theta}}.
\end{align*}
Integrating $\mathbf{u}_j$ yields,
\begin{align}\label{u_int}
    \int_{t}^{t+T} \mathbf{u}_j(\tau)d\tau = \int_{t}^{t+T}Y_2(\tau)\boldsymbol{\theta}d\tau.
\end{align}

Since the regressor matrix $Y_2$ depends on the derivative of the error signal, we assume that the derivative of the error signal is measurable. Moreover, at the steady-state the matrix $Y_2$ translates to $\begin{bmatrix}
        \mathbf{i}_j & -\omega_g J \mathbf{i}_j
    \end{bmatrix}$. Expressing $Y_2$ in terms of the error signal and the reference current translates to $Y_2 = \begin{bmatrix}
        (\mathbf{e}_j+\mathbf{i}_{{ref}_j}) & -\omega_g J (\mathbf{e}_j+\mathbf{i}_{{ref}_j})
    \end{bmatrix} $.

\begin{remark}
    In applications/situations where the derivative of the error signal is not readily measurable, an approximation in terms of designing a derivative filter captures the behavior of the derivative of the error signal and is given as follows
    \begin{align*}
    \mathbf{\dot{s}}_{f_j} = -\sigma \mathbf{s}_{f_j} - \sigma^2 \mathbf{e}_j, &\hspace{4mm}
    \mathbf{e}_{f_j} = \mathbf{s}_{f_j} + \sigma \mathbf{e}_{j}.
\end{align*}
For a choice of sufficiently high gain $\sigma$, the developed filter provides $\mathbf{e}_{f_j} \approx \mathbf{\dot{e}}_j$.
\end{remark}

\begin{theorem}\label{thm_2}
   Consider the closed loop error dynamics in \eqref{closed-loop_ed} together with the update law in \eqref{param_update}. If the reference input signals are persistently exciting according to
   \begin{align}\label{excit_conditions}
       \int_{t}^{t+T} i_{{ref}_{jd}}^2 (\tau)d\tau > \gamma, \text{ and } \int_{t}^{t+T} i_{{ref}_{jq}}^2 (\tau)d\tau > \gamma,
   \end{align}
   for some $T>0$ and $\gamma >0$, then the error dynamics in \eqref{closed-loop_ed} is exponentially stable in the sense that
  \begin{align*}
       \norm{\begin{bmatrix} \mathbf{e}_j(t) \\ \boldsymbol{\tilde{\theta}}(t) \end{bmatrix}} \leq \zeta_1\norm{\begin{bmatrix} \mathbf{e}_j(0) \\ \boldsymbol{\tilde{\theta}}(0) \end{bmatrix}} e^{-\frac{\zeta_2}{2} t}
   \end{align*}
   where,
   \begin{align*}
       \zeta_1 = \sqrt{\frac{\max\{l_{g_j},1/\gamma_1\}}{\min\{l_{g_j},1/\gamma_1\}}}, \hspace{2mm}\zeta_2 = \frac{\min\{(r_{g_j} + k_1), \gamma\}}{\max\{l_{g_j},1/\gamma_1\}},
   \end{align*}
   $k_1$ is a tunable control gain and $\gamma_1$ is the parameter adaptation rate tuner. 
\end{theorem}
\begin{proof}
  Consider the following Lyapunov candidate function
    $$V(\mathbf{e}_j,\boldsymbol{\tilde{\theta}}) = \frac{l_{g_j}}{2}\mathbf{e}_j^\top \mathbf{e}_j + \frac{1}{2\gamma_1}\boldsymbol{\tilde{\theta}}^\top \boldsymbol{\tilde{\theta}},$$
    taking the first time-derivative along the variables and substituting \eqref{closed-loop_ed} yields
    \begin{align*}
        \dot{V} &=  {\mathbf{e}_j}^\top l_{g_j}\dot{\mathbf{e}}_j  - \frac{1}{\gamma_1}\boldsymbol{\tilde{\theta}}^\top \boldsymbol{\dot{\hat{\theta}}}, \\
         &= -(r_{g_j} + k_1)\mathbf{e}_j^\top \mathbf{e}_j + l_{g_j}\omega_g\mathbf{e}_j^\top J \mathbf{e}_j + \mathbf{e}_j^\top Y_1 \boldsymbol{\tilde{\theta}} - \frac{1}{\gamma_1}\boldsymbol{\tilde{\theta}}^\top \boldsymbol{\dot{\hat{\theta}}},
    \end{align*}
    from the observation $\mathbf{e}_j^\top J \mathbf{e}_j = 0$, and substituting the parameter update law in \eqref{param_update} yields,
    \begin{align} \nonumber       
    \dot{V} &= -(r_{g_j}+k_1)\mathbf{e}_j^\top\mathbf{e}_j + \mathbf{e}_j^\top Y \boldsymbol{\tilde{\theta}} - \boldsymbol{\tilde{\theta}}^\top Y^\top \mathbf{e}_j - \\\nonumber
           & \boldsymbol{\tilde{\theta}}^\top\left(\int_{t}^{t+T}Y_2(\tau) d\tau\right)^\top \int_{t}^{t+T}\bigg(\mathbf{u}_j(\tau) - Y_2(\tau)\bigg)\boldsymbol{\hat{\theta}}(t)d\tau.
    \end{align}
    Substituting $\int_{t}^{t+T} \mathbf{u}_j(\tau)d\tau$ from \eqref{u_int} yields,
    \begin{align}\nonumber
        \dot{V}&= -(r_{g_j} + k_1)\mathbf{e}_j^\top\mathbf{e}_j - \boldsymbol{\tilde{\theta}}\bigg(\int_{t}^{t+T}Y(\tau) d\tau\bigg)^\top \times \\\nonumber
        &\hspace{2.3cm} \int_{t}^{t+T}\bigg(Y_2(\tau) \boldsymbol{\theta} - Y_2(\tau) \boldsymbol{\hat{\theta}}\bigg)d\tau, \\ \label{V_dot_bound}
        &\leq -(r_{g_j} + k_1)\mathbf{e}_j^\top \mathbf{e}_j - \boldsymbol{\tilde{\theta}}^\top \int_{t}^{t+T}Y_2(\tau)^\top Y_2(\tau)d\tau \boldsymbol{\tilde{\theta}}.
    \end{align}
If there exists $T>0$ and $\gamma>0$ such that $\int_{t}^{t+T}Y_2(\tau)^\top Y_2(\tau)d\tau \ge \gamma I$, 
   \begin{align*}
    \nonumber    \dot{V} &\leq -(r_{g_j} + k_1)\norm{\mathbf{e}_j}^2 - \gamma \norm{\boldsymbol{\tilde{\theta}}}^2 \\
  \nonumber    \frac{d}{dt}\bigg(\norm{\begin{bmatrix} \mathbf{e}_j \\ \boldsymbol{\tilde{\theta}} \end{bmatrix}}^2\bigg)  &\leq -\underbrace{\frac{\min\{(r_{g_j} + k_1), \gamma\}}{\max\{l_{g_j},1/\gamma_1\}}}_{\zeta_2}\norm{\begin{bmatrix} \mathbf{e}_j \\ \boldsymbol{\tilde{\theta}} \end{bmatrix}}^2.
    \end{align*}
    Integrating and using \textit{comparison lemma} yields,
\begin{align*}
   \norm{\begin{bmatrix} \mathbf{e}_j(t) \\ \boldsymbol{\tilde{\theta}}(t) \end{bmatrix}} \leq \underbrace{\sqrt{\frac{\max\{l_{g_j}, 1/\gamma_1\}}{\min\{l_{g_j}, 1/\gamma_1\} }}}_{\zeta_1}\norm{\begin{bmatrix} \mathbf{e}_j(0) \\ \boldsymbol{\tilde{\theta}}(0) \end{bmatrix}} e^{-\frac{\zeta_2}{2} t}.
\end{align*}
Now, to show that $\int_{t}^{t+T}Y_2(\tau)^\top Y_2(\tau)d\tau > 0$, it is sufficient to show that both its trace and determinant are positive: 
\begin{align*}
    &\textsf{trace}\bigg(\int_{t}^{t+T}Y_2(\tau)^\top Y_2(\tau)d\tau\bigg) \\
    &\hspace{1cm}= (1+\omega_g^2)\int_{t}^{t+T}(e_{jd}(\tau)+ i_{{ref}_{jd}}(\tau))^2d\tau \\
    &\hspace{1.5cm}+ (1+\omega_g^2)\int_{t}^{t+T}(e_{jq}(\tau) + i_{{ref}_{jq}}(\tau))^2d\tau,
\end{align*}
and
\begin{align*}
    &\textsf{det}\bigg(\int_{t}^{t+T}Y_2(\tau)^\top Y_2(\tau)d\tau\bigg) \\
    &\hspace{2mm}= \omega_g^2\bigg(\int_{t}^{t+T}(e_{jd}(\tau) +i_{{ref}_{jd}}(\tau))^2 d\tau \\
    &\hspace{2cm}+\int_{t}^{t+T}(e_{jq}(\tau) + i_{{ref}_{jq}}(\tau))^2d\tau\bigg)^2\\
    &\hspace{2mm}- \omega_g^2\bigg(\int_{t}^{t+T}(e_{jq}(\tau) + i_{{ref}_{jq}}(\tau))^2d\tau\\
    &\hspace{2.2cm}- \int_{t}^{t+T}(e_{jd}(\tau) + i_{{ref}_{jd}}(\tau))^2d\tau \bigg)^2 \\
    &=4\omega_g^2 \bigg(\int_{t}^{t+T}(e_{jd}(\tau) + i_{{ref}_{jd}}(\tau))^2 d\tau\bigg)\times \\
    &\hspace{3cm} \bigg(\int_{t}^{t+T}(e_{jq}(\tau) + i_{{ref}_{jq}}(\tau))^2d\tau\bigg).
\end{align*}

 \underline{\textbf{Claim:}} The trace and determinant are strictly positive if there exists $T<\infty$ such that $\int_{t}^{t+T}(e_{jd}(\tau)+ i_{{ref}_{jd}}(\tau))^2d\tau > 0$ and $\int_{t}^{t+T}(e_{jq}(\tau)+ i_{{ref}_{jq}}(\tau))^2d\tau > 0$.

\underline{\textbf{Proof of Claim:}} 
Suppose that $\int_{t}^{t+T}(e_{jd}(\tau)+ i_{{ref}_{jd}}(\tau))^2d\tau  = 0$ for some $T>0$, then from \eqref{V_dot_bound}, it follows that 
\begin{align*}
    \dot{V} \leq -(r_{g_j} + k_1) \norm{\mathbf{e}_j}^2,
\end{align*}
which implies that $\dot{V}$ is negative semi-definite. Thus, $V \in \mathcal{L}_{\infty}$ which implies $\mathbf{e}_j \in \mathcal{L}_{\infty}$. Since $\mathbf{i}_{{ref}_j}$ is assumed to be bounded, it implies that $\mathbf{i}_j \in \mathcal{L}_{\infty}$. Consequently, the control input $\mathbf{v}_j \in \mathcal{L}_{\infty}$. Integrating $\dot{V}$ yields,
\begin{align*}
        V(\infty)-V(0) \leq -(r_{g_j} + k_1) \int_{0}^{\infty}\norm{\mathbf{e}_j}^2 dt.
\end{align*}
It follows that $\mathbf{e}_j \in \mathcal{L}_2$. From the implications, $\mathbf{e}_j$ is uniformly continuous. Thus, invoking Barbalat's lemma it follows that $\mathbf{e}_j\longrightarrow \underline{\mathbf{0}}$. Consequently, for sufficiently large $t_1$, $\int_{t_1}^{t_1+T}(\left\|\mathbf{e}_j(\tau)\right\|^2)^2d\tau  = 0$, which implies that $\int_{t_1}^{t_1+T}\mathbf{e}_{jd}^2(\tau)d\tau  = 0$, which implies that $\int_{t_1}^{t_1+T}i_{{ref}_{jd}}^2(\tau)d\tau  = 0$. This contradicts the hypothesis in  \eqref{excit_conditions}. 

\begin{figure}[h!] 
\centerline{\includegraphics[width=.8\textwidth]{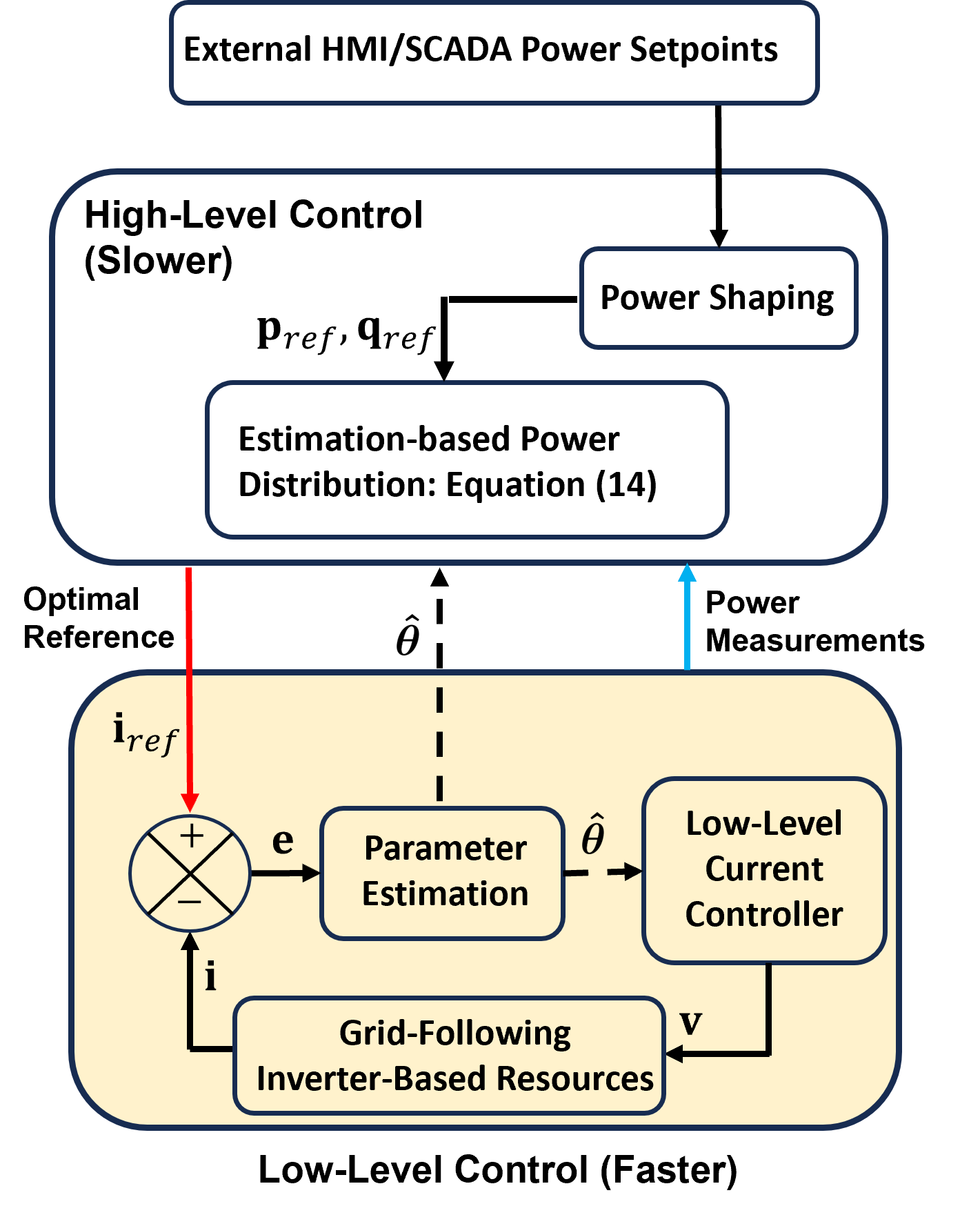}}
\caption{Structure of the proposed concurrent controller and estimator}
\label{control_struct}
\end{figure}

Moreover, similar arguments could be made with the assumption $\int_{t}^{t+T}(e_{jq}(\tau)+ i_{{ref}_{jq}}(\tau))^2d\tau  = 0$ leading to the contradiction $\int_{t_1}^{t_1+T}i_{{ref}_{jq}}^2(\tau)d\tau  = 0$.
This concludes the proof of the claim and also concludes the proof of the theorem.
\end{proof}

The result in Theorem \ref{thm_2} provides the conditions for the global exponential stability of the dynamics in \eqref{closed-loop_ed} and convergence of the current tracking error to zero exponentially. Moreover, the parameter estimation error $\boldsymbol{\tilde{\theta}}$ is shown to be bounded and exponentially convergent under the condition that the system is persistently excited. However, recent developments in \cite{5717148}, \cite{https://doi.org/10.1002/acs.2945} have proposed the use of concurrent learning adaptive control to guarantee exponential convergence of both tracking error and the parameter estimation error. Pivotal in establishing the exponential convergence relied upon their assumption on the availability of the portion of the regressor matrix data beforehand, thus relaxing the requirement of persistent excitation to finite excitation. However, this approach cannot be applied to the work in this paper since the current reference $\mathbf{i}_{{ref}_j}$ is generated from a high-level optimizer and cannot be known prior. Therefore, we demonstrate the positive definiteness of the matrix $\int_{0}^{t} Y_2(\tau)Y_2(\tau)^\top d\tau$ to establish the convergence of the parameter estimation error through persistent excitation, as shown in Fig. \ref{parameter_adapt}. Fig.~\ref{parameter_adapt} shows the response of the parameter estimator. 

\begin{figure}[t!] 
\centerline{\includegraphics[width=.98\textwidth]{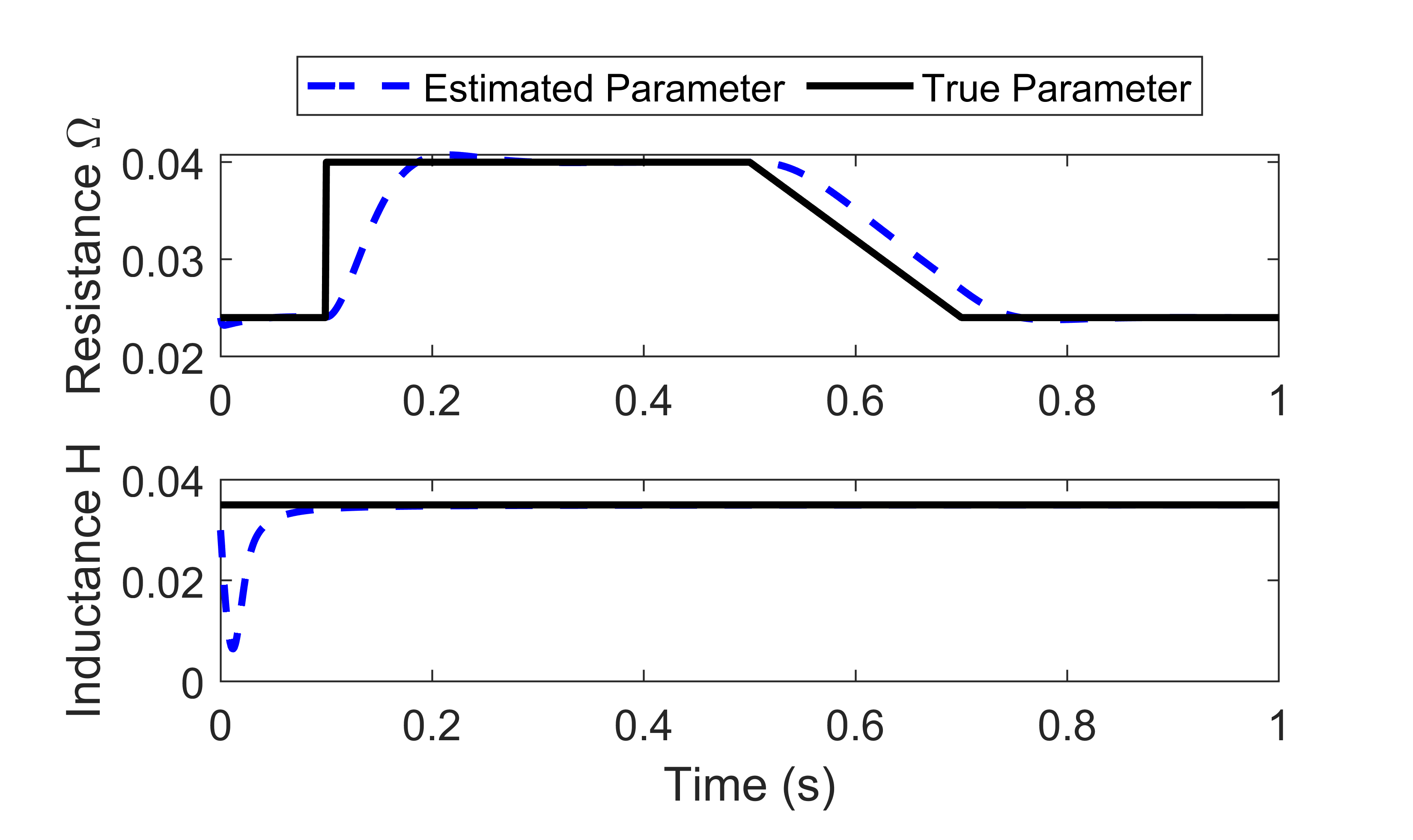}}
\caption{Parameter Estimator Performance}
\label{parameter_adapt}
\end{figure}

\section{Reference Shaping and Estimate-based Power Distribution}\label{ref-shaping}
Next, we present the high-level controller, which consists of reference shaping and optimal estimate-based power distribution. Assuming a higher-level HMI/SCADA operator sets the active and reactive power setpoints, the reference-shaping filter ensures the setpoints adhere to the IBRs apparent power limits.
\begin{remark}
    The feasible operating active and reactive power setpoints generated from the HMI/SCADA are characterized by a closed Euclidean ball given as: 
    \begin{equation}
        \mathcal{B}(s) = \{\mathbf{x} = [p,q]^\top \in \mathbb{R}^2: \norm{\mathbf{x}}_2 \leq s \},
    \end{equation}
    that is, the set of references $(p_{ref},q_{ref})$ always lies inside the ball or at most on the boundary, where $s \in \mathbb{R}$ is the total apparent power.
\end{remark}

\begin{figure}[t!] 
\centerline{\includegraphics[width=.98\textwidth]{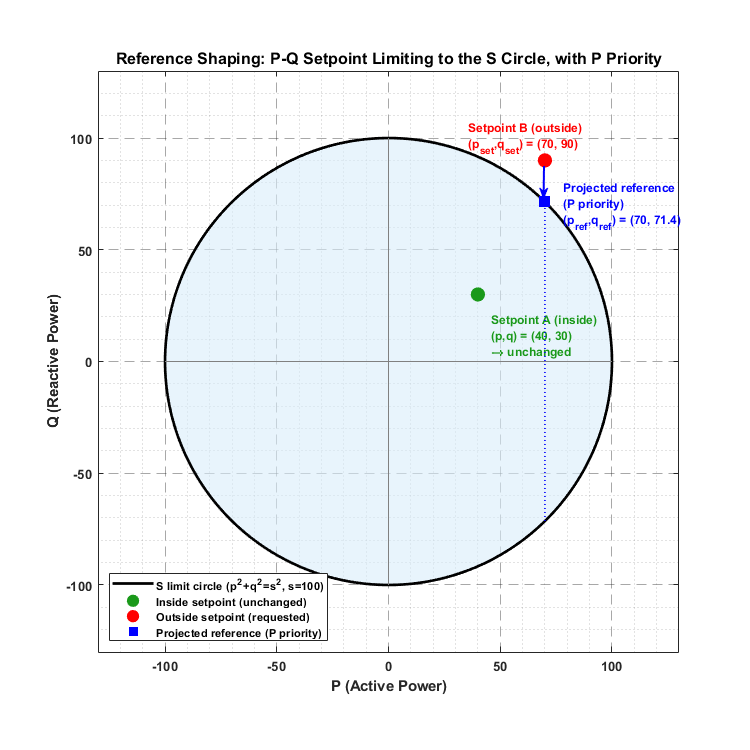}}
\caption{Power Circle}
\label{power_circle}
\end{figure}

Consider the optimization problem:

\begin{equation}
    \begin{aligned}
        \Minimize_{\mathbf{p},\mathbf{q}}& \frac{1}{2}\norm{\begin{bmatrix} \mathbf{p} \\ \mathbf{q} \end{bmatrix}}_2^2 \\
        \SubjectTo & \begin{bmatrix} \mathbf{g}(\hat{r}_{g_j})  & \underline{\mathbf{0}}^\top \\ \underline{\mathbf{0}}^\top & \mathbf{g}(\hat{r}_{g_j}) \end{bmatrix} \begin{bmatrix} \mathbf{p} \\ \mathbf{q} \end{bmatrix} = \begin{bmatrix} p_{ref} \\ q_{ref} \end{bmatrix},
    \end{aligned}
\end{equation}
where $\mathbf{p} \triangleq \left[\begin{array}{cccc} {p_{1},p_{2},\hdots,p_{N}} \end{array}\right]^\top \in \mathbb{R}^N$, $(p_{ref}, q_{ref}) \in \mathbb{R}$, are the reference shaped active and power references, $\mathbf{q} \triangleq \left[\begin{array}{cccc} {q_{1},q_{2},\hdots,q_{N}} \end{array}\right]^\top \in \mathbb{R}^N$. \\ $\mathbf{g}(\hat{r}_{g_j}) \triangleq \left[\begin{array}{cccc} {g(\hat{r}_{g_1}),\hdots,g(\hat{r}_{g_N}}) \end{array}\right] \in \mathbb{R}^{1 \times N}$ and,
\begin{equation*}
    g(\hat{r}_g) = \begin{cases} 0, \text{if} \hspace{2mm} |r_{g_0} - \hat{r}_g| \geq \epsilon 
    \\ 1, \text{if} \hspace{2mm} |r_{g_0} - \hat{r}_g| < \epsilon\end{cases}
\end{equation*}
If the parameter estimate deviates from the nominal known value, then that IBR is excluded from the power distribution.

\begin{figure}[h!] 
\centerline{\includegraphics[width=.9\textwidth]{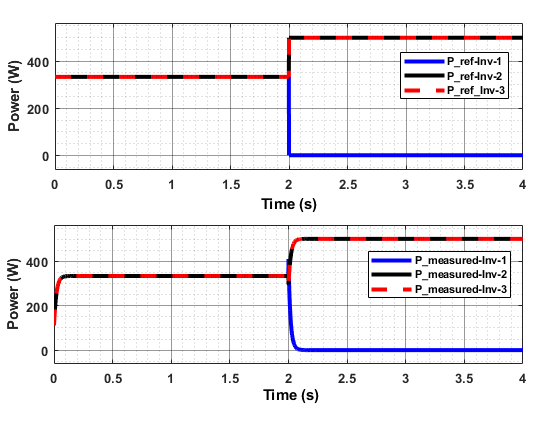}}
\caption{Power sharing between the three IBRs (top) and the power measurements (bottom) at the IBR terminals for each IBR}
\label{power_share}
\end{figure}

\section{Simulation}\label{Sim}
A three-IBR system/plant in an active power priority mode with a rated apparent power of 1000 \textsf{VA} is considered to assess the performance of the developed/proposed controller. The nominal grid voltage is $v_g = 392$\textsf{V} (LN rms), and the grid frequency $w_g=60$\textsf{Hz}. The nominal line resistance and inductance are chosen as follows: 0.027$\Omega$, 0.037\textsf{H}. A fault is injected at Inverter-1 at $t=2secs$. The setpoints from the HMI, $p_{setpoint}$, and $q_{setpoint}$ are set at 1000\textsf{W}, and 1000\textsf{VAR} respectively. Figure-\ref{power_shape} shows the reference shaping algorithm enforcing the active power priority by setting the active power reference ($p_{ref}$) to 1000\textsf{W}, while forcing the reactive power reference ($q_{ref}$) to 0\textsf{VAR} to meet the total plant rated apparent power of 1000\textsf{VA}.

\begin{figure}[h!] 
\centerline{\includegraphics[width=.9\textwidth]{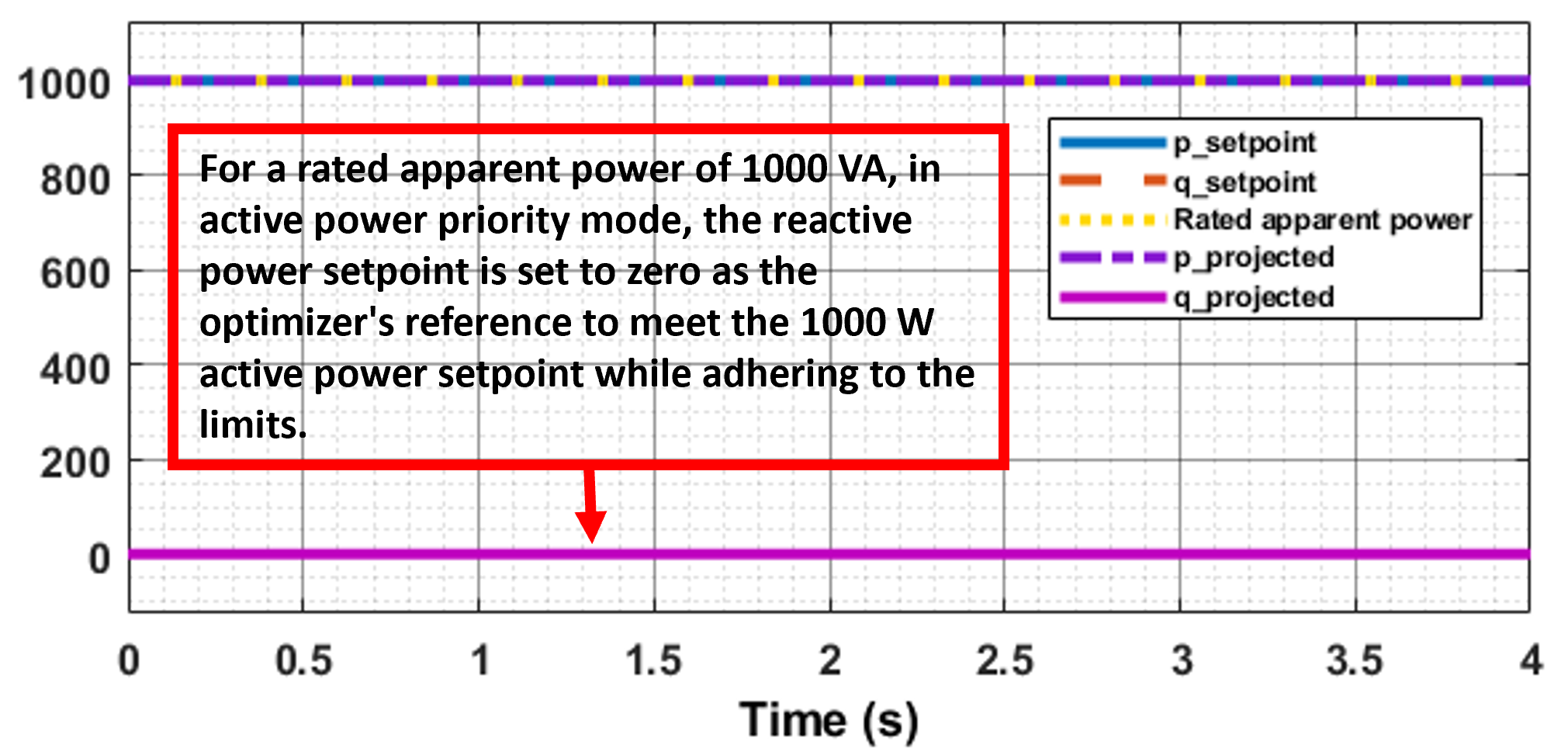}}
\caption{Reference shaping based on the apparent power of the plant.}
\label{power_shape}
\end{figure}

\begin{figure}[t!] 
\centerline{\includegraphics[width=.9\textwidth]{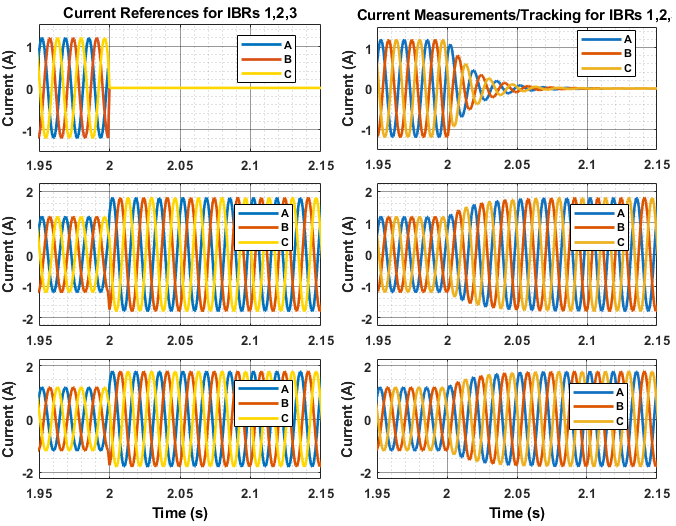}}
\caption{Current sharing between the three IBRs. The plots to the left show the current reference generated for each IBR to track. The plots on the right show the current measurement at the IBR terminal. It can be seen that the IBR currents are readjusted within $0.1secs$ of the event occurrence.}
\label{current_share}
\end{figure}

Since the reactive power reference is zero, Figure-\ref{power_share} shows the active power reference split between the three IBRs and the consequent active power measurements at the terminals of the IBRs. It can be seen that at $t=2secs$, when the fault occurs, the power is split between the remaining two active IBRs, and the IBR with the fault is excluded from the power sharing contribution.

In Figure-\ref{current_share}, the plots of the left show the current references for each IBR in three-phase generated based on the active and reactive power setpoints for each IBR. It can be seen that for IBR-1, the current goes to zero at $t=2secs$, while simultaneously the other two IBRs increase their current references to maintain the active power reference. The plots to the right show the current measurement at the IBR terminals. It can be seen that the current is split and the steady state is attained within 0.1\textsf{secs} of the fault occurrence.

\section{Conclusion}\label{Conc}
The proposed framework integrates an online parameter estimator, optimal power allocation, and current control for GFL IBRs. The online estimator enables the IBRs to track grid variations and accordingly adjust their current contribution. The estimated parameters are incorporated into the optimization framework to obtain the grid-aware active and reactive power setpoints, which are further converted to current references. The numerical simulations show the improved response of the GFL IBRs during dynamic grid events. Overall, the results validate the effectiveness of the concurrent estimation and current control framework for effective IBR operation. Future studies leave scope for the impact of measurement noise and parameter estimation errors resulting in false-positive errors.

\bibliographystyle{IEEEtran}
\bibliography{myreferences}

\end{document}